\documentclass[onefignum,onetabnum]{siamart171218}

\usepackage{lipsum}
\usepackage{amsfonts}
\usepackage{graphicx}
\usepackage{epstopdf}
\usepackage{algorithmic}
\ifpdf
  \DeclareGraphicsExtensions{.eps,.pdf,.png,.jpg}
\else
  \DeclareGraphicsExtensions{.eps}
\fi

\newsiamremark{remark}{Remark}
\newsiamremark{hypothesis}{Hypothesis}
\crefname{hypothesis}{Hypothesis}{Hypotheses}
\newsiamthm{claim}{Claim}

\headers{A Prox Operator for Multispectral Phase Retrieval}{B. Roig-Solvas, L. Makowski, and D. H. Brooks}

\title{A Proximal Operator for Multispectral Phase Retrieval Problems}

\author{Biel Roig-Solvas\thanks{Department of Electrical and Computer Engineering, Northeastern University, Boston, MA
  (\email{biel@ece.neu.edu}, \email{brooks@ece.neu.edu}).}
\and Lee Makowski\thanks{Department of Bioengineering, Northeastern University, Boston, MA
  (\email{l.makowski@neu.edu}).}
\and Dana H. Brooks\footnotemark[1]}

\usepackage{amsopn}

\ifpdf
\hypersetup{
  pdftitle={A Prox Operator for Multispectral Phase Retrieval},
  pdfauthor={B. Roig-Solvas, L. Makowski, and D. H. Brooks}
}
\fi

\begin{document}

\maketitle

\begin{abstract}
 Proximal algorithms have gained popularity in recent years in large-scale and distributed optimization problems. One such problem is the phase retrieval problem, for which proximal operators have been proposed recently. The phase retrieval problem commonly refers to the task of recovering a target signal based on the magnitude of linear projections of that signal onto known vectors, usually under the presence of noise. A more general problem is the multispectral phase retrieval problem, where sums of these magnitudes are observed instead. In this paper we study the proximal operator for this problem, which appears in applications like X-ray solution scattering. We show that despite its non-convexity, all local minimizers are global minimizers, guaranteeing the optimality of simple descent techniques. An efficient linear time exact Newton method is proposed based on the structure of the problem's Hessian. Initialization criteria are discussed and the computational performance of the proposed algorithm is compared to that of traditional descent methods. The studied proximal operator can be used in a distributed and parallel scenarios using an ADMM scheme and allows for exploiting the spectral characteristics of the problem's measurement matrices, known in many physical sensing applications, in a way that is not possible with non-splitted optimization algorithms. The dependency of the proximal operator on the rank of these matrices, instead of their dimension, can greatly reduce the memory and computation requirements for problems of moderate to large size ($N>10^4$) when these measurement matrices admit a low-rank representation.
\end{abstract}

\begin{keywords}
  Proximal Operators, Phase Retrieval, ADMM
\end{keywords}


\section{Introduction}\label{Section:Intro}
Proximal algorithms have gained popularity in recent years due to their ability to solve large-scale and distributed optimization problems efficiently \cite{Parikh2014}. When the problem to be solved can be cast as a collection of additive terms sharing a common variable $x$, one can reformulate it as a \textit{consensus} optimization problem\cite{Parikh2014,Boyd2010}:
\begin{equation}
\begin{split}
    \operatorname*{min}_{x}\; \sum_t^T f_t(x)\quad\equiv\quad \operatorname*{min}_{x}\;  \sum_t^T f_t(x_t)\quad\text{s.t.}\quad x_t = z \quad \forall t
\end{split}
\label{consensus}
\end{equation}
where each additive term depends on its own variable $x_t$ and each of these variables is forced to agree with the consensus variable $z$. This formulation can be iteratively solved in a distributed manner using an Alternate Direction Method of Multipliers (ADMM) \cite{Boyd2010}, where at each iteration a proximal (\textit{prox}) operator of the form \ref{prox} is solved, together with an update on the consensus and additional dual variables.
\begin{equation}
     \textit{prox}_{f}(u) = \operatorname*{min}_{x}\; f(x) + ||x-u||_2^2.
     \label{prox}
\end{equation}
In this work we explore how to solve problems of the type~\ref{prox} stemming from applying ADMM to problems of the form:
\begin{equation}
    \operatorname*{min}_{y\in\mathbb{C}^M}\; \sum_t^T \left(\left(A_t y\right)^H\left(A_t y\right) - b_t\right)^2
    \label{orig_problem}
\end{equation}
for $A_t\in \mathbb{C}^{M\times K_t}$ and $b_t \in R_{+}$ for $t = 1,\dots,T$, which can be seen as an instance of \ref{consensus} where $f_t = \left(\left(A_t y\right)^H\left(A_t y\right) - b_t\right)^2$.\\
\\
Problems like \ref{orig_problem} arise, among other places, in multispectral phase retrieval applications. Phase retrieval commonly refers to the problem of recovering of a complex signal $y \in  \mathbb {C}^M$ from the measurement of the squared projections onto the vectors $a_i\in  \mathbb {C}^{M}$, i.e. find $y$ such that $|<a_i, y>|^2 = b_i$ for all $i = 1,\dots,Q$ \cite{Shechtman2015}. Despite being a classic problem in the signal processing world \cite{Gerchberg1972,Fienup1978,Fienup1982}, the field of phase retrieval has experienced an increase in interest in recent years, thanks to the development of algorithms relying on  semi-definite programming \cite{Candes2011, Candes2013, Waldspurger2013, Fogel2016}, Wirtinger flow \cite{Candes2014, Chen2017} or sparse reconstructions \cite{Li2012,Shechtman2014}.\\
\\
A more general problem than classic phase retrieval arises when the measurements $b_t$ are not single projections, but the sum of several of squared projections, i.e. $b_t = \sum_k^{K_t}\,\left(<a_{t,k},y>\right)^2$. These multispectral phase retrieval problems arise in applications like X-ray solution scattering \cite{Konarev2003} and fiber diffraction \cite{Roig-Solvas2017} , where the contributions of many spectral components sum incoherently in the detector.  In this work we present an efficient algorithm to solve the prox operator of each of the additive terms in \ref{orig_problem}: 
\begin{equation}
     \textit{prox}_{f}(w) = \operatorname*{min}_{y\in\mathbb{C}^M}\; \left(\left(A y\right)^H\left(A y\right) - b\right)^2 + ||y-w||_2^2.
     \label{prox_f}
\end{equation}
Since this problem must be solved in turn for each term in the consensus sum, efficiency becomes of particular importance.

Unlike many of the widely used prox operators \cite{Combettes2011}, \ref{prox_f} can't be solved in closed form except in very special cases (e.g. $K_i = 1$ or $A_i = {I}$ \cite{Soulez2017}) and in general needs to be solved iteratively. To that end, in this work we present an efficient algorithm to globally optimize \ref{prox_f}. We prove that, despite its non-convexity, all minimizers of \ref{prox_f} are equivalent in objective function value and moreover, under mild conditions, \ref{prox_f} has a unique minimizer. We propose a Newton algorithm to reach any of these global minimizers and exploit the structure of \ref{prox_f} to compute each of the Newton iterates in linear time (as opposed to the general $\mathcal{O}(N^3)$ cost), resulting in a very computationally efficient approach to minimize \ref{prox_f}.\\
\\
The structure of the paper is as follows. In Section 2 we recast \ref{prox_f} into a problem over the reals and prove the equivalence of all minimizers and the uniqueness of the minimizer in the case $w>0$. In Section 3 we present the proposed algorithm to solve \ref{prox_f}, together with guidelines to select a good starting point for the iterative algorithm. In Section 4 we report results from numerical simulations comparing the performance of the proposed approach to standard implementations of gradient descent and Newton method and  analyze the sensitivity of the algorithm performance to the spectral and norm properties of $A_i$ and $w$. Finally, in Section 5 we discuss the implications of the proposed method and future work.

\section{Global Optimization of the Prox Operator}\label{Section:Proof}

In this section we prove the existence of a unique minimum for the prox operator presented in \ref{prox_f}. We start by recasting \ref{prox_f} as a problem of the form:
\begin{equation}
    \operatorname*{min}_{x\in R^N}\quad \left( x^Tx - b\right)^2 + ||x-u||_{\Sigma}^2
    \label{P1}
    \tag{P1}
\end{equation}
to which we will refer as \ref{P1} for the rest of this work.
\subsection{Recasting}
We start the recasting by noting that \ref{prox_f} has an equivalent representation over the reals:
\begin{equation}
    \operatorname*{min}_{y\in \mathbb{C}^M}\;\left(\left(A y\right)^H\left(A y\right) - b\right)^2+||y-w||_2^2 \quad \equiv  \operatorname*{min}_{y_R, y_I\,\in R^M}\;\left(\begin{bmatrix}
       y_R\\ 
       y_I
     \end{bmatrix}^T Q \begin{bmatrix}
       y_R\\ 
       y_I
     \end{bmatrix} - b\right)^2+||\begin{bmatrix}
       y_R\\ 
       y_I
     \end{bmatrix} - \begin{bmatrix}
       w_R\\ 
       w_I
     \end{bmatrix}||_2^2 
\end{equation}
where $y_R$ and $y_I$ are the real and imaginary parts of $y$, $w_R$ and $w_I$ are the real and imaginary parts of $w$ and the symmetric matrix $Q$ is of the form:
\begin{equation}
    Q = \begin{bmatrix}
      A_R^T A_R + A_I^T A_I & 0\\ 
       0 & A_R^T A_R + A_I^T A_I 
     \end{bmatrix}
     \label{matrixQ}
\end{equation}
for $A_R = \mathbf{Re}\left(A\right)$ and $A_I = \mathbf{Im}\left(A\right)$. Taking the singular value decomposition (SVD) of $Q$: $Q = U_Q \Sigma_Q U_Q^T$ and writing:
\begin{equation}
    \tilde{y} = U_Q^T\begin{bmatrix}
       y_R\\ 
       y_I
     \end{bmatrix} \quad \tilde{w} = U_Q^T\begin{bmatrix}
       w_R\\ 
       w_I
     \end{bmatrix}
\end{equation}
we get the optimization problem:
\begin{equation}
    \operatorname*{min}_{y_R, y_I\,\in R^M}\;\left(\begin{bmatrix}
       y_R\\ 
       y_I
     \end{bmatrix}^T Q \begin{bmatrix}
       y_R\\ 
       y_I
     \end{bmatrix} - d\right)^2+||\begin{bmatrix}
       y_R\\ 
       y_I
     \end{bmatrix} - \begin{bmatrix}
       w_R\\ 
       w_I
     \end{bmatrix}||_2^2 \quad \equiv  \operatorname*{min}_{\tilde{y}}\quad \left( \tilde{y}^T\Sigma_Q \tilde{y} - d\right) + ||\tilde{y}-\tilde{w}||_2^2
\end{equation}
We assume in the following that $\Sigma_Q$ is full rank. If that were not the case, the components of $\tilde{x}$ on the nullspace of $\Sigma_Q$ can be trivially set to $\tilde{w}$ for the optimum solution. Setting $N = 2M$ and denoting $x \in R^N = S\Sigma_Q^{1/2}\tilde{y}$, $u \in R^N= S\Sigma_Q^{1/2}\tilde{w}$ and $\Sigma = \Sigma_Q^{-1}$, where $S = \text{diag}\left(\text{sign}\left(\Sigma_Q^{1/2}\tilde{w}\right)\right)$ we get the final form:
\begin{equation}
    \operatorname*{min}_{\tilde{y}}\quad \left( \tilde{y}^T\Sigma_Q \tilde{y} - b\right) + ||\tilde{y}-\tilde{w}||_2^2 \quad \equiv \quad  \operatorname*{min}_{x\in R^N}\quad \left( x^Tx - b\right)^2 + ||x-u||_{\Sigma}^2
\end{equation}
with $\Sigma \succ 0$ and $u\geq 0$.

\subsection{Proof of Global Optimality}
Starting from \ref{P1}, the structure of the proof is as follows: first we present an equivalent problem to \ref{P1} over $R^{2\,N}$ (which we will refer to as \ref{P2}) with a convex objective function and $N$ non-convex quadratic equality constraints. These equality constraints are then relaxed to convex inequality constraints, resulting in the convex relaxation \ref{P2C} of the problem \ref{P2}. In Lemma \ref{LemmaRelax} we show that the proposed convex relaxation is exact, i.e. optimizing \ref{P2C} is equivalent to minimizing \ref{P2}, and thus \ref{P1}. \\
\\
In Lemmas \ref{Lemma1} and \ref{Lemma2} we discuss the locations of the minimizers of \ref{P1}. More specifically, in Lemma \ref{Lemma1} we show that if $w>0$, then all minimizers are in the non-negative orthant, while in Lemma \ref{Lemma2} we tackle the case when $w$ is non-negative but not positive, where minimizers can exist outside the non-negative orthant. We show that each of them will have the same objective value as another minimizer in the non-negative orthant. In Lemma \ref{Lemma3} we use the KKT conditions for \ref{P2C} to show that any minimizer in the non-negative orthant of \ref{P1} has a one to one correspondence with a local minimizer of \ref{P2C} with the same objective value.\\
\\
Finally, in Theorem \ref{Thm1} we show that all minimizers of \ref{P1} are equivalent in objective value and thus all minimizer are global. To do so, we use the convexity of \ref{P2C} and Lemma \ref{Lemma3} to show that all minimizers in the non-negative orthant of \ref{P1} share the same objective value and using Lemmas \ref{Lemma1} and \ref{Lemma2} we extend this proposition to all minimizers of \ref{P1}. Corollary \ref{Corollary1} provides a stronger statement in the case of $u>0$, where we show that in that case the minimizer of \ref{P1} is unique.\\
\\
We start the proof by stating the optimization problems \ref{P2} and \ref{P2C}. First we unravel all the terms in \ref{P1}:
\begin{equation}
    \operatorname*{min}_{x\in R^N}\quad \left( x^Tx\right)^2 + b^2 - 2 b\left( x^Tx\right) + x^T \Sigma x + u^T \Sigma u -2 u^T \Sigma x 
\end{equation}
Next we add the redundant variable $z \in R^N$ such that $z_i = \left(x_i\right)^2$. Under this constraint, we have that $x^T x = \mathbf{1}^T z$ and $x^T \Sigma x = \sigma^T z$, where $\sigma$ is the vector of diagonal values of the matrix $\Sigma$. Applying these changes and dropping the terms that are constant with respect to $x$ ans $z$, we get \ref{P2}:
\begin{equation}
\begin{split}
    \operatorname*{min}_{x,z\in R^N}\quad & z^T \left(\mathbf{1} \mathbf{1}^T\right) z + \left( \sigma - 2 b \mathbf{1}\right)^T z - 2 u^T \Sigma x \\
    & \text{s.t.}\quad z_i = \left(x_i\right)^2\quad \forall i
    \end{split}
    \label{P2}
    \tag{P2}
\end{equation}
We note that this objective is quadratic in $z$ and linear in $x$ and thus convex. The convex relaxation \ref{P2C} is given by replacing the equalities $z_i = \left(x_i\right)^2$ by convex inequalities $z_i \geq \left(x_i\right)^2$, resulting in the relaxed problem:
\begin{equation}
\begin{split}
    \operatorname*{min}_{x,z\in R^N}\quad & z^T \left(\mathbf{1} \mathbf{1}^T\right) z + \left( \sigma - 2 b \mathbf{1}\right)^T x - 2 u^T \Sigma x \\
    & \text{s.t.}\quad z_i \geq \left(x_i\right)^2\quad \forall i
    \end{split}
    \label{P2C}
    \tag{P2C}
\end{equation}
The positive-definiteness of $\Sigma$ and the non-negativity of $u$ guarantee that this relaxation is exact, as shown in the following Lemma:
\begin{lemma}
\ref{P2C} has a minimizer that satisfies the constraints $z_i \geq \left(x_i\right)^2\; \forall i$ with equality. If $u>0$, this minimizer is unique.  
\label{LemmaRelax} 
\end{lemma}
 
\begin{proof}
To prove this it suffices to show that given the linear relationship between $x$ and $-2 u^T \Sigma$ and the non-positivity of $-2 u^T \Sigma$ by construction, increasing $x$ can never increase the objective function. For those indices $i$ for which $u_i>0$, increasing $x_i$ will always decrease the objective and thus $x_i^*$ reaches its upper bound, resulting in $z_i^* = (x_i^*)^2$. For the rest of the indices, for which $u_i = 0$, $x_i$ plays no role in \ref{P2C} and we only require it to be feasible, leading to $x_i^* \in [-\sqrt{z_i^*}, \sqrt{z_i^*}]$. As a result, the pair $[x^*,z^*]$ with $z_i^* = (x_i^*)^2$ is a minimizer of \ref{P2C}.
\end{proof}

Next we discuss the location of the minimizers of \ref{P1} in $R^N$. We start with the case $u>0$:

\begin{lemma}
If $u>0$, all minimizers of \ref{P1} are non-negative.
\label{Lemma1} 
\end{lemma}
\begin{proof}
Assume by contradiction that there exists an $x$ with an index $k$ such that $x_k < 0$ and $x$ locally minimizes \ref{P1}, i.e. $\nabla f_{P_{1}}(x) = 0$ and $\nabla^2 f_{P_{1}}(x) \succeq 0 $. The stationarity condition equates to:
\begin{equation}
    \left(4({x}^T {x} - b)\; I + 2\Sigma\right){x^*} = 2\Sigma u \quad \iff \quad \left(2({x}^T {x} - b) + \sigma_i\right)x_i = \sigma_i u_i \quad \forall i
    \label{stationarity}
\end{equation}
Given that $x_k < 0$ and $\sigma_k u_k >0$ by assumption, we have that $\left(2({x}^T {x} - b) + \sigma_k\right)<0$. Next we analyze the Hessian of \ref{P1} at $x$. This Hessian presents a diagonal plus rank 1 structure given by:
\begin{equation}
    \nabla^2 f_{P_{1}}(x) =  \left(4({x}^T {x} - b)\; I + 2\Sigma\right) + 8 x x^T
\end{equation}
Writing the diagonal component as $D = \left(4({x}^T {x} - b)\; I + 2\Sigma\right)$ and its eigenvalues as $d_i$ with $d_i\leq d_{i+1}$, and the eigenvalues of $\nabla^2 f_{P_{1}}(x)$ as $\lambda$ with $\lambda_i \leq \lambda_{i+1}$, we have that \cite{Golub1973}:
\begin{equation}
    d_i \leq \lambda_i \leq d_{i+1}.
    \label{Interlacing}
\end{equation}
We remind the reader that the diagonal entries of $\Sigma$ are the inverse of the eigenvalues of $Q$. Given the block structure of $Q$ in Equation \ref{matrixQ}, one can see that all eigenvalues of $Q$ will have a multiplicity of at least 2. It follows that the eigenvalues of $\Sigma$, and thus of $D$, will also have multiplicity 2 or greater. Combining this with the fact that the lowest eigenvalue of $D$ is negative by assumption (as there exists an index $k$ for which $\left(2({x}^T {x} - b) + \sigma_k\right)<0$, i.e. $D$ is diagonal and one of its entries is negative), we have that at least 2 eigenvalues of $D$ will be negative, i.e $d_1 = d_2 < 0$.\\
\\
By \ref{Interlacing}, the smallest eigenvalue of $\nabla^2 f_{P_{1}}(x)$ will be bounded above by $d_2$, resulting in $\lambda_1 <0$. It follows then that $\nabla^2 f_{P_{1}}(x) \not \succeq 0 $, which contradicts the assumption that $x$ is a local optimum of \ref{P1}. All stationary points with $x_k < 0$ must then be either saddle points or local maximizers, as the Hessian at those points will be either indefinite or negative semi-definite. As a result, any minimizer of \ref{P1} must have $x\geq 0$.
\end{proof}
For the case when $u$ has some entries equal to 0, we get the weaker proposition:
\begin{lemma}
If $u\geq0$ and $u\not> 0$, then any local minimizer $x$ with an index $k$ such that $x_k<0$ has an equivalent minimizer $\hat{x}: f_{P_{1}}(x) = f_{P_{1}}(\hat{x})$ in the non-negative orthant of the form $\hat{x} = |x|$
\label{Lemma2} 
\end{lemma}

\begin{proof}
As in Lemma 1, $x$ must satisfy $\nabla f_{P_{1}}(x) = 0$ and $\nabla^2 f_{P_{1}}(x) \succeq 0 $ to be a local minimizer of \ref{P1}. For each $k$ for which $x_k<0$, we must have that $u_k = 0$. Otherwise, to satisfy the stationarity condition, one would have that $\left(2({x}^T {x} - b) + \sigma_k\right)<0$ and, as shown in Lemma 1, that implies the existence of negative eigenvalues of the Hessian, contradicting the positive-semidefiniteness assumption.\\
\\
Since $u_k = 0$ for each $x_k < 0$, it is easy to verify that $f_{P_{1}}(x) = f_{P_{1}}(\hat{x})$, as the function becomes invariant to changing the sign of $x_k$, and that the stationarity condition and the semidefinitieness of the Hessian are also trivially satisfied, concluding the proof.
\end{proof}
Next we tie the minimizers of \ref{P1} with those of \ref{P2C}. We do so in the following Lemma:
\begin{lemma}
If $x\geq 0$ is a minimizer of \ref{P1}, then the pair $[x,z]$ with $z_i = x_i^2$ is a minimizer of \ref{P2C}.
\label{Lemma3} 
\end{lemma}
\begin{proof}
The pair $[x,z]$ must satisfy the KKT conditions of \ref{P2C} to minimize it. Writing the constraints in the standard form $g_i(x,z) \leq 0$ for $g_i(x,z) = x_i^2-z_i$, the KKT conditions for this problem are:
\begin{equation}
\begin{split}
 \nabla_{x} f_{P_{2C}}(x^*,z^*) = - \sum \mu_i^* \nabla_{x}g_i(x^*,z^*)\\ 
 \nabla_{z} f_{P_{2C}}(x^*,z^*) = - \sum \mu_i^* \nabla_{z}g_i(x^*,z^*)\\
\end{split}
\label{KKT-Stat}\tag{Stationarity}
\end{equation}
\begin{equation}
    g_i(x^*,z^*)\leq 0
    \label{KKT-PrimFeas}\tag{Primal Feasibility}
\end{equation}
\begin{equation}
    \mu_i^*\geq 0
    \label{KKT-DualFeas}\tag{Dual Feasibility}
\end{equation}
\begin{equation}
  \mu_i^*\,g_i(x^*,z^*) = 0   
    \label{KKT-ComplSlack}\tag{Comp. Slackness}
\end{equation}
Primal feasibility and complementary slackness are satisfied by assumption, as the equality $z_i = x_i^2$ is equivalent to stating $g_i(x,z) = 0$ for all $i$. Developing the stationarity condition for $y$ yields:
\begin{equation}
    \begin{split}
        \nabla_{x} f_{P_{2C}}(x^*,z^*) &= - \sum \mu_i^* \nabla_{x}g_i(x^*,z^*)\\ 
            -2\Sigma u &=  -2\, \text{diag}(\mu^*) x^*\\
            \Sigma u &= \text{diag}(x^*) \mu^*
    \end{split}
\end{equation}
As $x$ is a minimizer of \ref{P1}, it must satisfy the stationarity condition \ref{stationarity} for \ref{P1}. Combining both stationarity conditions for \ref{P1} and \ref{P2C} results in $\mu_i = \left(2({x}^T {x} - b) + \sigma_i\right)$. The assumption $\nabla^2 f_{P_{1}}(x) \succeq 0 $ necessarily implies that $\left(2({x}^T {x} - b) + \sigma_i\right)\geq 0$, as in Lemma \ref{Lemma1}, so we are guaranteed that the choice $\mu_i = \left(2({x}^T {x} - b) + \sigma_i\right)$ satisfies both the stationarity condition for $x$ and the dual feasibility condition in \ref{P2C}. The only condition left to check is the stationarity with respect to $z$. Developing the stationarity condition for $z$ and using the above $\mu$ identity we get:
\begin{equation}
    \begin{split}
         \nabla_{z} f_{P_{2C}}(x^*,z^*) &= - \sum \mu_i^* \nabla_{z}g_i(x^*,z^*)\\
         2\left(\mathbf{1}\mathbf{1}^T\right) z^* + \left(\sigma - 2 b \mathbf{1}\right) &= \mu^* \\
         \text{diag}(x^*)\; (2\left(\mathbf{1}\mathbf{1}^T\right) z^* + \left(\sigma - 2 b \mathbf{1}\right)) &= \text{diag}(x^*)\;\mu^* \\
         \text{diag}( 2\left(\mathbf{1}^T z^* - b \right) \mathbf{1} + \sigma) x^* &= \Sigma u\\
         \left(2({x^*}^T {x^*} - b)\; I + \Sigma\right){x^*} &= \Sigma u 
    \end{split}
\end{equation}
The last equation is equivalent to requiring $y$ to be a stationarity point of \ref{P1}, which it is by assumption, so we can conclude that the pair $[x,z]$ satisfies all the KKT conditions of \ref{P2C} and is thus a minimizer of \ref{P2C}, which finishes the proof.
\end{proof}
Combining Lemmas \ref{Lemma1}, \ref{Lemma2} and     \ref{Lemma3} and the convexity of \ref{P2C}, we postulate that:
\begin{theorem}
Every local minimizer of \ref{P1} is a global minimizer.  
\label{Thm1} 
\end{theorem}
\begin{proof}
By Lemma \ref{Lemma3}, all local minimizers in the non-negative orthant of \ref{P1} are local minimizers of \ref{P2C} and by convexity of \ref{P2C}, all local minimizers of \ref{P2C} must have equal objective value. As $f_{P_{2C}}(x,z) = f_{P_{1}}(x)$ whenever $x\geq 0$ and $z_i = x_i^2$, it follows that all local minimizers in the non-negative orthant of \ref{P1} will also have the same objective function value and thus will be global minimizers of \ref{P1}. Finally, by \ref{Lemma2}, any minimizer outside of the non-negative orthant will have the same objective value as at least one minimizer in the non-negative orthant, which extends the status of global minimizers to all local minima of \ref{P1}.  
\end{proof}
\begin{corollary}
If $u>0$, the minimizer is unique.
\label{Corollary1}

\end{corollary}
\begin{proof}
We prove this corollary by proving a more general statement, i.e. if $u>0$, there exists only one stationary point in the non-negative orthant. By contradiction, assume that there exist two different points $\hat{x}$ and  $\tilde{x}$, $\hat{x}\neq \tilde{x}$, such that they both are stationary points of \ref{P1}, i.e. $\nabla f_{P_{1}}(\hat{x}) = 0$  and $\nabla f_{P_{1}}(\tilde{x}) = 0$. Assume without loss of generality that $||\hat{x}||_2^2 \geq ||\tilde{x}||_2^2$. The stationarity condition gives us the following characterization of $\hat{x}$ and $\tilde{x}$:
\begin{equation}
    \hat{x}_i = \frac{\sigma_i u_i}{\left(2({\hat{x}}^T {\hat{x}} - b) + \sigma_i\right)},\quad \tilde{x}_i = \frac{\sigma_i u_i}{\left(2({\tilde{x}}^T {\tilde{x}} - b) + \sigma_i\right)},\quad \forall i
\end{equation}
Given the non-negativity of the left-hand sides of these expressions and the positivity of the numerators of the right-hand sides, it follows that the denominators can't be negative. The condition $||\hat{x}||_2^2 \geq ||\tilde{x}||_2^2$ can be broken into the conditions: $||\hat{x}||_2^2 = ||\tilde{x}||_2^2$ and $||\hat{x}||_2^2 > ||\tilde{x}||_2^2$. For the former, we note that if $||\hat{x}||_2^2 = ||\tilde{x}||_2^2$, then the right-hand sides of the expressions above are equal, and thus we have that $\hat{x} = \tilde{x}$, which contradicts the initial assumption that $\hat{x}$ and $\tilde{x}$ are different. For the latter, we have that:
\begin{equation}
     ||\hat{x}||_2^2 > ||\tilde{x}||_2^2 \implies \tilde{x}_i = \frac{\sigma_i u_i}{\left(2({\tilde{x}^T} {\tilde{x}} - b) + \sigma_i\right)}>  
     \frac{\sigma_i u_i}{\left(2({\hat{x}}^T {\hat{x}} - b) + \sigma_i\right)} =  \hat{x}_i \implies ||\hat{x}||_2^2 < ||\tilde{x}||_2^2
\end{equation}
which is a contradiction. It follows thus that $\hat{x}$ and  $\tilde{x}$ cannot be stationary points unless they are equal, which proves that uniqueness of the non-negative stationary point. Given that, by Lemma 1, when $u>0$ all minimizers of \ref{P1} are in the non-negative orthant, and that \ref{P1} must have a minimizer because it is bounded below, at least one minimizer must exist in the non-negative orthant. As every minimizer is a stationary point and that stationary point is unique, it follows that the minimizer of \ref{P1} must also be unique. 
\end{proof}

\section{An Efficient Solver for the Prox Operator}\label{Section:Algo}

The equivalence of a all local minimizers allows us to use simple descent methods to reach the global solution of \ref{P1}. Between the three equivalent problems \ref{P1}, \ref{P2} and \ref{P2C}, we focus on \ref{P1} due to its simplicity, as it lacks any constraints and the size of its variable space is $N$ as opposed to $2N$. We begin by computing the gradient $\nabla_{x}$ and the Hessian $\nabla^2_{x}$ of \ref{P1}:
\begin{equation}
    \nabla_{x} \; = \; \left(4({x}^T {x} - b)\; I\right){x} + 2\Sigma \left(x-u\right)
    \label{Gradient}
\end{equation}
\begin{equation}
    \nabla^2_{x} \; = \; 8x x^T + 4({x}^T {x} - b)\; I + 2\Sigma 
    \label{Hessian}
\end{equation}
The Hessian matrix in \ref{Hessian} can be separated into a (full-rank) diagonal component and a rank 1 matrix. Under this structure, the inverse of $\nabla^2_{x}$ can be computed in linear time using the Sherman-Morrison formula to perform rank 1 updates on the inverse of matrix \cite{Golub2012}:
\begin{equation}
    \left(A + u v^T \right)^{-1} = A^{-1} - \frac{A^{-1} u v^T A^{-1}}{1+v^TA^{-1}u} 
    \label{S-M} \tag{S-M}
\end{equation}
This property allows us to add second-order information in our optimization at the same cost as a first-order gradient descent. We propose to minimize \ref{P1} using the Newton method together with the \ref{S-M} update:
\begin{equation}
    \begin{split}
        x^{k+1} &= x^{k} - \alpha \left(\nabla^2_{x^{k}}\right)^{-1} \nabla_{x^{k}}\\
        &= x^{k} - \alpha \left(8x^{k} {x^{k}}^T + 4({x^{k}}^T {x^{k}} - b)\; I + 2\Sigma \right)^{-1} \nabla_{x^{k}}\\
    \end{split}
\end{equation}
Introducing the vector $\xi^{k} = 2\sigma + 4({x^{k}}^T {x^{k}} - b)\mathbf{1}$ containing the diagonal elements of $\nabla^2_{x^K}$, the element-wise quotient operator $\oslash$ and the mapping $diag(x)$ that maps from $R^N$ vectors to $R^{N\times N}$ diagonal matrices, we get that:
\begin{equation}
    \begin{split}
        x^{k+1} &= x^{k} - \alpha \left(diag\left(\mathbf{1}\oslash\xi^k\right) - 8\frac{diag\left(\mathbf{1}\oslash\xi^k\right) x^{k} {x^{k}}^T diag\left(\mathbf{1}\oslash\xi^k\right)}{1+{x^{k}}^Tdiag\left(\mathbf{1}\oslash\xi^k\right){x^{k}}}    \right) \nabla_{x^{k}}\\
        &= x^{k} - \alpha \left(\nabla_{x^{k}}\oslash\xi^k -   \frac{\left(8{x^k}^T\left(\nabla_{x^{k}}\oslash\xi^k\right)\right)}{1+8{x^{k}}^T \left(x^k\oslash\xi^k\right)}  \left({x^k}\oslash\xi^k\right)  \right)
    \end{split}
\end{equation}
which involves only element-wise vector computations and can be computed in linear time.\\
\\
Next we explore two additional factors of the optimization of \ref{P1}: the choice of the starting iterate $x^0$ and the size of the descent direction $\alpha^k$. For the choice of the first iterate, we study the structure of the optimizer $x^*$ and develop element-wise and norm bounds for $x^*$ that we can apply to $x^0$. We start by setting the gradient in Equation \ref{Gradient} to 0:



\begin{equation}
     \left(4({x^{*}}^T {x^{*}} - b)\; I\right){x^{*}} = 2\Sigma \left(u-x^{*}\right)
     \label{GradientOpti}
\end{equation}
We assume from now on that ${x^{*}}\geq 0$. As shown in Section \ref{Section:Proof}, either there is a unique optimizer for \ref{P1}, which is non-negative, or there are multiple equivalent minimizers, for which one of them is non-negative, so we focus our discussion on norm and element bounds on the non-negative case. Under that assumption, all the elements on the left hand side (LHS) of \ref{GradientOpti} share the same sign as ${x^{*}}^T {x^{*}} - b$. Assuming the LHS to be positive, i.e. ${x^{*}}^T {x^{*}} > b$, leads to $u>x^*$, which in turn leads to $u^T u > {x^{*}}^T {x^{*}} > b$. Similarly, when the LHS is negative we get $b > {x^{*}}^T {x^{*}} > u^T u$ and $x^*>u$. We see that these conditions rely on the relationship of $b$ and $u$, leading to:
\begin{equation}
    \begin{split}
        \text{If}\;u^T u>b\quad & \implies \quad u^T u>{x^{*}}^T{x^{*}}>b\;\text{and}\; u_i>x^{*}_i\\
        \text{If}\;b>u^T u\quad & \implies \quad b>{x^{*}}^T{x^{*}}>u^T u\;\text{and}\; x^{*}_i>u_i\\
    \end{split}
    \label{Bounds1}
\end{equation}
Analyzing equation \ref{GradientOpti} element-wise, we get:
\begin{equation}
    x_i^* = \frac{\sigma_i\;u_i}{\sigma_i + 2\left({x^{*}}^T{x^{*}}-b\right)}
    \label{x_opti}
\end{equation}
As the LHS has been shown to be non-negative and the numerator in the RHS is non-negative by construction, it follows that the RHS denominator must be non-negative too. This results in the following lower bound on the norm of ${x^{*}}$:
\begin{equation}
        \sigma_i + 2\left({x^{*}}^T{x^{*}}-b\right) \geq 0\quad \implies \quad
        {x^{*}}^T{x^{*}} \geq b - \frac{\sigma_i}{2} \; \forall i \quad \implies \quad  {x^{*}}^T{x^{*}} \geq b - \frac{\sigma_{-}}{2}
        \label{Bounds2}
\end{equation}
where $\sigma_{-}$ is the minimum diagonal entry of $\Sigma$. This bound is only useful for the case $b>u^T u$, as otherwise Equation \ref{Bounds1} gives a tighter bound ${x^{*}}^T{x^{*}}>b$. Applying this inequality to Equation \ref{x_opti} we derive an element-wise upper bound on $x^*$:
\begin{equation}
    x_i^* = \frac{\sigma_i\;u_i}{\sigma_i + 2\left({x^{*}}^T{x^{*}}-b\right)} \leq \frac{\sigma_i\;u_i}{\sigma_i -\sigma_{-}}
    \label{Bounds3}
\end{equation}
We summarize these bounds in the following table:
\begin{center}
 \begin{tabular}{|c ||c  | c||} 
 \hline
 Case & Norm Bounds & Element-wise Bounds  \\ [0.5ex] 
 \hline\hline
$u^Tu>b$ & $u^T u>{x^{*}}^T{x^{*}}>b$ &  $u_i>x_i$  \\ 
 \hline
 $b > u^Tu$ & $b>{x^{*}}^T{x^{*}}> \text{max}(u^T u, b - \frac{\sigma_{-}}{2})$ & $\frac{\sigma_i}{\sigma_i -\sigma_{-}}u_i>x_i>u_i$  \\
 \hline
\end{tabular}
\label{BoundsTable}
\end{center}
Regarding the choice of the step size $\alpha$, we propose an alternative to the trivial choice $\alpha = 1$ that relies on computing the optimal step size in the descent direction at each iteration. We use the fact that \ref{P1} is a quartic problem and thus optimizing it over a line involves the minimization of a univariate quartic polynomial:
\begin{equation}
\begin{split}
    \alpha^*\; &= \; \operatorname*{argmin}_{\alpha \in R}\quad \left( \left({x^k} + \alpha \Delta_{x^k}\right)^T\left({x^k} + \alpha \Delta_{x^k}\right) - b\right)^2 + ||\left({x^k} + \alpha \Delta_{x^k}\right)-u||_{\Sigma}^2\\
    &= \; \operatorname*{argmin}_{\alpha \in R}\quad a_4\,\alpha^4\,+a_3\,\alpha^3\,+a_2\,\alpha^2\,+a_1\,\alpha\,+a_0
    \end{split}
    \label{Eq:OptimalAlpha}
\end{equation}
for appropriate values of $a_{0},\dots,a_4$. The optimal $\alpha$ can be obtained by differentiating the quartic polynomial in \ref{Eq:OptimalAlpha} and keeping the real negative root with the smallest absolute value of the resulting cubic, which can be found via a closed form expression. 
\section{Numerical Experiments}\label{Section:Numerical}
In this section we compare the proposed method to solve \ref{P1} with conventional implementations of the Newton Method and gradient descent. All three methods are implemented with both a unit step update and the optimal step-length update presented in Section \ref{Section:Algo} and are tested with a random initialization as well as with a warm start using the bounds derived in Section \ref{Section:Algo}. Pseudo-code for the tested methods is shown in Algorithm \ref{Algo1}:\\
\\
\begin{algorithm}[H]
\caption{Pseudo-code for the tested methods}\label{Algo1}
\begin{algorithmic}[1]
 \STATE{Sample $u$, $\Sigma$ and $x^0$ as per Algorithm \ref{Algo2}}\\
 \IF{Warm Start}
 \STATE{$x^0 = u \,\sqrt{\frac{b}{u^T u}}$}
 \ENDIF
 \STATE{ $\text{tol} = 10^{-6}$}
 \STATE{$\text{max-iter} = 5\cdot10^{4}$}
 \STATE{$k = 0$}
\WHILE{$\left(\nabla_{x^k}^T \nabla_{x^k}\right)>\text{tol} \quad  \&\&  \quad \text{k}<\text{max-iter}$}
\IF{Gradient Descent}
       \STATE{     $\Delta_{x^k}= -\nabla_{x^k}$
        }
        \ELSIF{Newton}
        \STATE{
            $ \Delta_{x^k} = -\left(\nabla_{x^k}^2\right)^{-1}\nabla_{x^k}$
        }
        \ELSIF{S-M Newton}
        \STATE{
            $\xi^{k} = 2\sigma + 4({x^{k}}^T {x^{k}} - b)\mathbf{1}$}
            \STATE{
            $\Delta_{x^k} = - \left(\nabla_{x^{k}}\oslash\xi^k -   \frac{\left(8{x^k}^T\left(\nabla_{x^{k}}\oslash\xi^k\right)\right)}{1+8{x^{k}}^T \left(x^k\oslash\xi^k\right)}  \left({x^k}\oslash\xi^k\right)  \right)$
        }
    \ENDIF
    
    \IF{Unit Step}
    \STATE{
             $\alpha = 1$
        }
        \ELSIF{Optimal Step}
        \STATE{
            $ \alpha = \operatorname*{min}_{\alpha} \;  f\left({x^k} + \alpha \Delta_{x^k}\right)\quad$ 
            for $f(x) = \left(x^T x - b\right)^2 + \frac{\rho}{2}(x-{u})^T\Sigma(x-{u})$
        }
    \ENDIF
   \STATE{ $ x^{k+1} = x^{k} + \alpha \Delta_{x^k}$}
    \STATE{$k = k+1$}
 \ENDWHILE

\end{algorithmic}
\end{algorithm}

The methods were tested for 20 different values of problem dimension $N$, ranging from $N=10$ to $2000$, spaced in logarithmically uniform intervals. For each $N$, the methods were tested for 50 Monte Carlo simulations with different sampled values of $u$, $\Sigma$ and $x^0$.\\ 
\\
The geometry of \ref{P1} depends, among other things, on the relationship between $b$ and $||u||_2^2$ and the condition number and Frobenius norm of $\Sigma$. To analyze the sensitivity of the proposed method to these factors, $u$ was sampled from a uniform distribution and scaled to having a squared norm ranging from 1/10 to 10 times that of $b$, while  $\Sigma$ was sampled to have linearly decaying eigenvalues of multiplicity 2, with a squared Frobenius norm ranging from  1/10 to 10 times that of $b$ and a condition number ranging from $1$ to $1000$. Setting $b = 100$, we get the following sampling scheme:\\ 
\\
\begin{algorithm}[H]
\begin{algorithmic}
 \STATE{$p\,\sim\,\mathcal{U}\left(0,3\right),\,q,r_1,r_2\,\sim\,\mathcal{U}\left(1,3\right)$}
 \STATE{$s_1,s_2\,\sim\,\mathcal{U}\left(0,1\right)^N$}
 \STATE{$b = 100$}
 \STATE{$t = \{ t\in R^\frac{N}{2}:\, t_i = 1 + \frac{i-1}{\frac{N}{2}-1}10^p \quad \forall \,i\in [1,\frac{N}{2}]\} $}
 \STATE{$\sigma = [t,t]\,\sqrt{\frac{10^q}{||[t,t]||_2^2}}$}
 \STATE{$\Sigma = \text{diag}\left(\sigma\right)$}
 \STATE{$u = s_1\,\sqrt{\frac{10^{r_1}}{||s_1||_2^2}}$}
 \STATE{$x^0 = s_2\,\sqrt{\frac{10^{r_2}}{||s_2||_2^2}}$}
 \end{algorithmic}
\caption{Pseudo-code for Monte Carlo sampling}\label{Algo2}
\end{algorithm}

\begin{figure}
    \centering
    \includegraphics[width = 12cm]{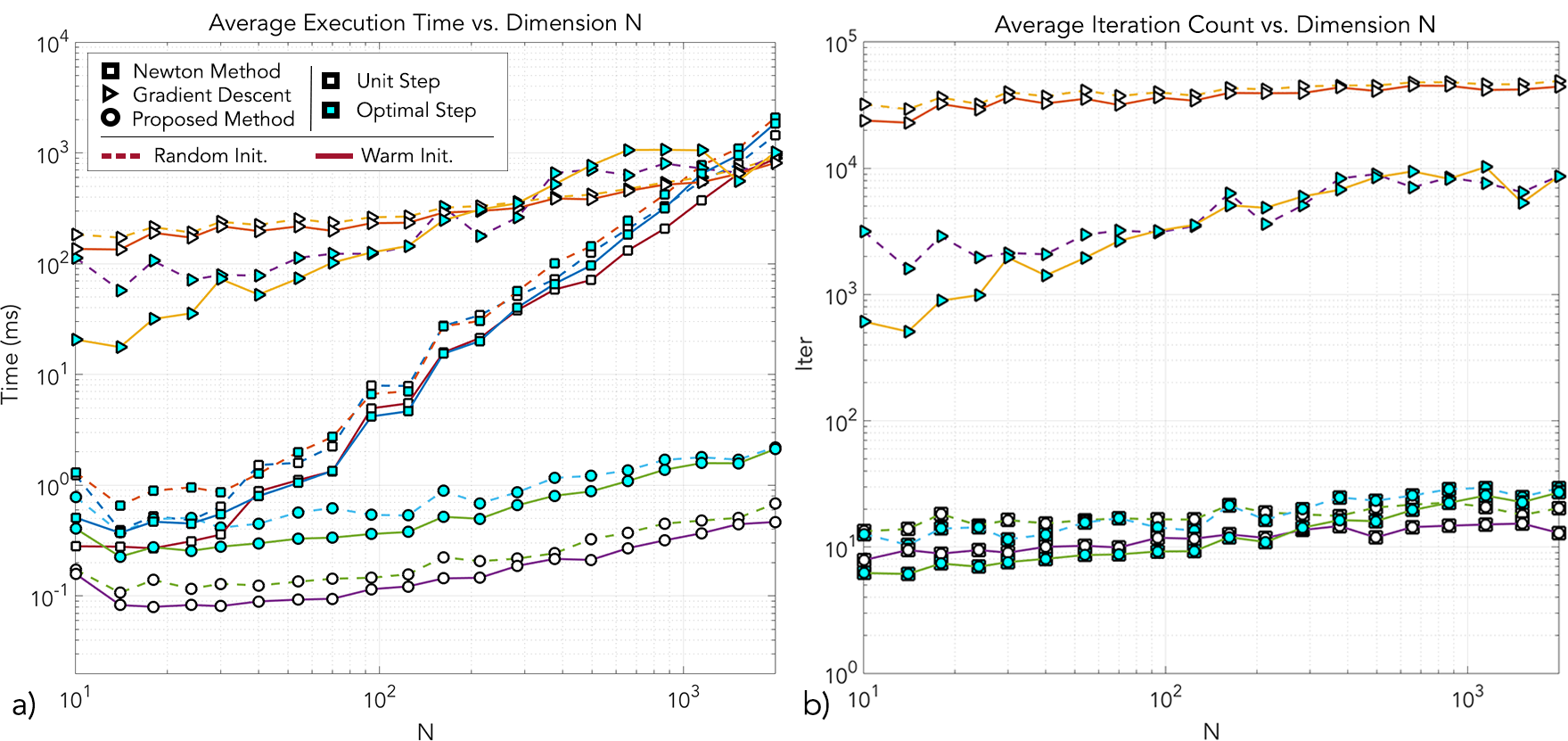}
    \caption{Average execution time and iteration count vs. dimension $N$ across methods: Each data point is the average of 50 Monte Carlo Simulations as explained above. The different methods are represented by the marker shapes (square, triangle and circle for Newton Method, Gradient Descent and the proposed S-M Newton method, respectively), while the step-size is shown in the color of the marker (white for unit step, cyan for optimal step). Finally, dotted lines show experiments run with a random initial point, while solid lines denote those initialized with a warm start, as explained in the text. }
    \label{fig:Results1}
\end{figure}

The results of the Monte Carlo simulations are presented in Figure \ref{fig:Results1}, where execution time and iteration count are shown as a function of the problem dimension $N$ in a log-log plot. The proposed approach, in circular markers, outperforms both gradient descent and standard Newton in all modalities. The computational complexity of the compared approaches becomes evident in Figure \ref{fig:Results1}.a, where the Newton method (in square markers) presents a much steeper increase in complexity due to the $O(N^3)$ cost of the Hessian inversion, compared to the linear complexity of the other two methods.\\
\\
The use of warm start, denoted by the plots in solid lines in Figure \ref{fig:Results1}, produces both lower execution time and iteration count in all cases when compared to the random start, in dotted lines. Regarding the choice of step size, the use of the optimal step size reduces the number of iterations in almost all cases, as shown in Figure \ref{fig:Results1}.b. However, the effect on execution time varies between methods: for the proposed method, the use of a unit step achieves a faster converge than the optimal step, due to the extra computation present in the latter approach, while for the Newton method, where each iterate requires a costly matrix inversion, the extra cost of computing the optimal step is negligible in comparison and the Newton method converges more rapidly to the optimizer when using the optimal step size.\\
\\
The difference in convergence rates between gradient descent and the Newton methods are evident in Figure \ref{fig:Results1}.b. In fact, most of the executions for the gradient descent experiments (90\% and 60\% for random and warm start, respectively, and $N=2000$) stopped due to the iteration limit set in Algorithm \ref{Algo1}, before the tolerance was reached, as can be seen by the progressive approach of both gradient descent curves in Figure \ref{fig:Results1}.b towards the 50000 iteration ceiling,  further illustrating the slow convergence rate of gradient descent for this problem.\\

\begin{figure}
    \centering
    \includegraphics[width = 14cm]{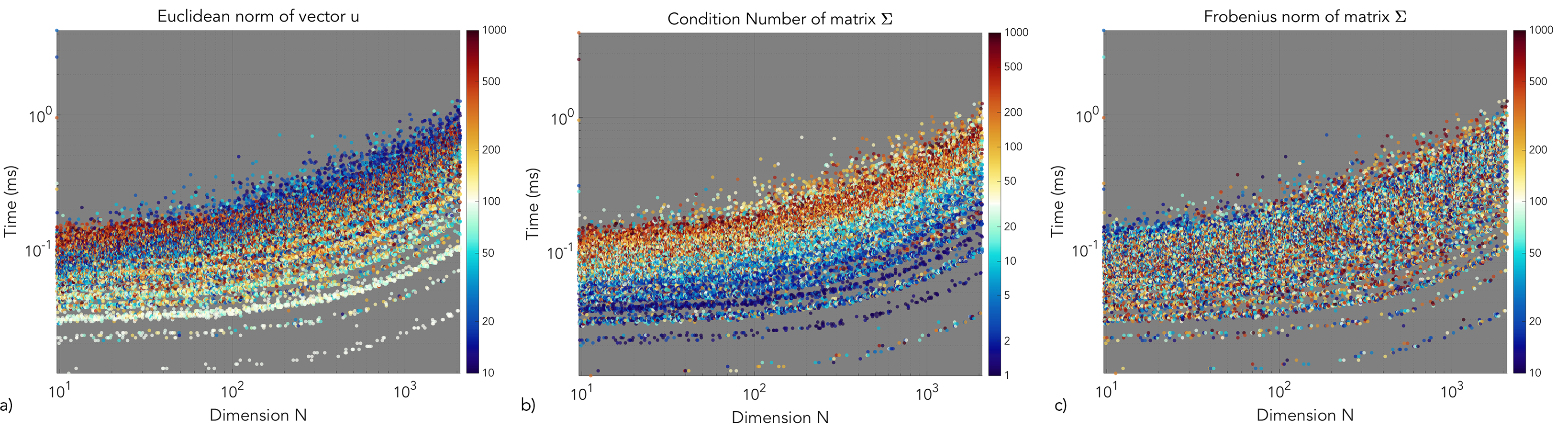}
    \caption{Effect of $||u||_2^2$, $\text{cond}(\Sigma)$ and $||\Sigma||_F^2$ on the running time of the proposed method. 50 values of $N$ were selected, spaced in log-uniform intervals between $N=10$ and $N=2000$. 500 experiments were carried out for each $N$, sampling $u$ and $\Sigma$ at random as in Algorithm \ref{Algo2}. The proposed method was run for each experiment, using the unit step and the warm initialization proposed in \ref{Algo1}. Each experiment is plotted with respect to its dimension $N$ and its running time, where incremental displacements where added on the horizontal axis to improve visualization. Colors and colorbars indicate  $||u||_2^2$, $\text{cond}(\Sigma)$ and $||\Sigma||_F^2$ in panels a, b and c, respectively. }
    \label{fig:Results2}
\end{figure}

In Figure \ref{fig:Results2} we report on the effect on  execution time of the problem parameters $||u||_2^2$, $\text{cond}(\Sigma)$ and $||\Sigma||_F^2$. To do so, we  discretized the range $N = [10,2000]$  in 50 log-uniform steps and ran 500 experiments for each value of $N$, sampling $u$ and $\Sigma$ at random as in Algorithm \ref{Algo2}. All experiments were run with the fastest configuration found in the previous experiments, i.e. using the warm start presented in Algorithm \ref{Algo1} and unit step size. Figure \ref{fig:Results2} shows the execution time of each experiment as a function of dimension $N$, where each experiment is colored based on $||u||_2^2$, $\text{cond}(\Sigma)$ and $||\Sigma||_F^2$, respectively, in each panel.\\
\\
The execution time has a clear dependence on $||u||_2^2$, decreasing as $||u||_2^2$ approaches $b$. As shown in the bounds developed in Section \ref{Section:Algo}, $||x^*||_2^2$ is always bounded between $b$ and $||u||_2^2$, so it is to be expected that $x^0$ converges faster to $x^*$ as these bounds become tighter. The condition number also shows a clear correlation with the execution time, where  convergence is faster as $\Sigma$ approaches a multiple of the identity matrix (we note that \ref{P1} can be solved in closed form when $\Sigma = c I$ \cite{Soulez2017}). However, as evidenced by Figure \ref{fig:Results2}.c, no clear patterns appear in the interaction between $||\Sigma||_F^2$ and execution time. 

\section{Conclusions} \label{Section:Conclusions}
In this paper we have studied the general form of the prox operator \cite{Combettes2011} that arises from the multispectral phase retrieval problem \ref{orig_problem}. We have shown that all minimizers of the prox operator \ref{prox_f} are equivalent in objective value, guaranteeing the global optimality of any local minima reached by simple descent techniques despite the non-convexity of the problem. We then analyzed the structure of the Hessian of the prox operator and shown that it can be inverted in linear time due to its diagonal plus rank 1 structure. Thus, an exact Newton method can be implemented that is efficient both in the number of iterates needed to achieve optimality and in the cost of each of these iterates. We also studied the properties of the minimizer of \ref{prox_f} in order to derive initial iterates that are close to the optimal solution. We report on the performance of the proposed method and compare it to that for regular implementations of the gradient descent and Newton methods, and show that the proposed approach achieved superior performance in all experimental scenarios studied. 
\\

\section*{Acknowledgments}
We would like to thank Mario Sznaier for helpful discussions. This work was supported by the U.S. Department of Energy, Office of Science, Basic Energy Sciences (grant no. DE–SC0000997).

\bibliographystyle{siamplain}
\bibliography{bibliography.bib}

\begin{thebibliography}{10}

\bibitem{Boyd2010}
{\sc S.~Boyd, N.~Parikh, E.~Chu, B.~Peleato, and J.~Eckstein}, {\em
  {Distributed Optimization and Statistical Learning via the Alternating
  Direction Method of Multipliers}}, Foundations and Trends{\textregistered} in
  Machine Learning, 3 (2010), pp.~1--122,
  \url{https://doi.org/10.1561/2200000016}.

\bibitem{Candes2014}
{\sc E.~Candes, X.~Li, and M.~Soltanolkotabi}, {\em {Phase Retrieval via
  Wirtinger Flow: Theory and Algorithms}},  (2014),
  \url{https://doi.org/10.1109/TIT.2015.2399924},
  \url{http://arxiv.org/abs/1407.1065
  http://dx.doi.org/10.1109/TIT.2015.2399924},
  \url{https://arxiv.org/abs/1407.1065}.

\bibitem{Candes2011}
{\sc E.~J. Candes, Y.~Eldar, T.~Strohmer, and V.~Voroninski}, {\em {Phase
  Retrieval via Matrix Completion}},  (2011),
  \url{http://arxiv.org/abs/1109.0573}, \url{https://arxiv.org/abs/1109.0573}.

\bibitem{Candes2013}
{\sc E.~J. Cand{\`{e}}s, T.~Strohmer, and V.~Voroninski}, {\em {PhaseLift:
  Exact and stable signal recovery from magnitude measurements via convex
  programming}}, Communications on Pure and Applied Mathematics, 66 (2013),
  pp.~1241--1274, \url{https://doi.org/10.1002/cpa.21432},
  \url{https://arxiv.org/abs/1109.4499}.

\bibitem{Chen2017}
{\sc Y.~Chen and E.~J. Cand{\`{e}}s}, {\em {Solving Random Quadratic Systems of
  Equations Is Nearly as Easy as Solving Linear Systems}}, Communications on
  Pure and Applied Mathematics, 70 (2017), pp.~822--883,
  \url{https://doi.org/10.1002/cpa.21638},
  \url{https://arxiv.org/abs/1505.05114}.

\bibitem{Combettes2011}
{\sc P.~L. Combettes and J.~C. Pesquet}, {\em {Proximal splitting methods in
  signal processing}}, Springer Optimization and Its Applications, 49 (2011),
  pp.~185--212, \url{https://doi.org/10.1007/978-1-4419-9569-8_10},
  \url{https://arxiv.org/abs/0912.3522}.

\bibitem{Fienup1978}
{\sc J.~R. Fienup}, {\em {Reconstruction of an object from the modulus of its
  Fourier transform}}, Optics Letters, 3 (1978), p.~27,
  \url{https://doi.org/10.1364/OL.3.000027},
  \url{https://www.osapublishing.org/abstract.cfm?URI=ol-3-1-27},
  \url{https://arxiv.org/abs/78}.

\bibitem{Fienup1982}
{\sc J.~R. Fienup}, {\em {Phase retrieval algorithms: a comparison}},
  Appl.{\~{}}Opt., 21 (1982), pp.~2758�--2769.

\bibitem{Fogel2016}
{\sc F.~Fogel, I.~Waldspurger, and A.~D'Aspremont}, {\em {Phase retrieval for
  imaging problems}}, Mathematical Programming Computation, 8 (2016),
  pp.~311--335, \url{https://doi.org/10.1007/s12532-016-0103-0},
  \url{https://arxiv.org/abs/1304.7735}.

\bibitem{Gerchberg1972}
{\sc R.~W. Gerchberg and W.~O. Saxton}, {\em {A practical algorithm for the
  determination of phase from image and diffraction plane pictures}}, Optik, 35
  (1972), pp.~237--246, \url{https://doi.org/10.1070/QE2009v039n06ABEH013642},
  \url{http://ci.nii.ac.jp/naid/10010556614/}.

\bibitem{Golub1973}
{\sc G.~H. Golub}, {\em {Some Modified Matrix Eigenvalue Problems}}, SIAM
  Review, 15 (1973), pp.~318--334, \url{https://doi.org/10.1137/1015032},
  \url{http://epubs.siam.org/doi/abs/10.1137/1015032}.

\bibitem{Golub2012}
{\sc G.~H. Golub and C.~F. Van~Loan}, {\em Matrix computations}, vol.~3, JHU
  Press, 2012.

\bibitem{Konarev2003}
{\sc P.~V. Konarev, V.~V. Volkov, A.~V. Sokolova, M.~H.~J. Koch, and D.~I.
  Svergun}, {\em {PRIMUS: a Windows PC-based system for small-angle scattering
  data analysis}}, Journal of Applied Crystallography, 36 (2003),
  pp.~1277--1282, \url{https://doi.org/10.1107/S0021889803012779},
  \url{http://scripts.iucr.org/cgi-bin/paper?S0021889803012779}.

\bibitem{Li2012}
{\sc X.~Li and V.~Voroninski}, {\em {Sparse Signal Recovery from Quadratic
  Measurements via Convex Programming}},  (2012), pp.~1--15,
  \url{https://doi.org/10.1137/120893707},
  \url{http://arxiv.org/abs/1209.4785}, \url{https://arxiv.org/abs/1209.4785}.

\bibitem{Parikh2014}
{\sc N.~Parikh and S.~Boyd}, {\em {Proximal Algorithms}}, Foundations and
  Trends{\textregistered} in Optimization, 1 (2014), pp.~127--239,
  \url{https://doi.org/10.1561/2400000003},
  \url{http://www.nowpublishers.com/articles/foundations-and-trends-in-optimization/OPT-003},
  \url{https://arxiv.org/abs/1502.03175}.

\bibitem{Roig-Solvas2017}
{\sc B.~Roig-Solvas and L.~Makowski}, {\em {Calculation of the cross-sectional
  shape of a fibril from equatorial scattering}}, Journal of Structural
  Biology,  (2017), \url{https://doi.org/10.1016/j.jsb.2017.05.003},
  \url{http://linkinghub.elsevier.com/retrieve/pii/S1047847717300837}.

\bibitem{Shechtman2014}
{\sc Y.~Shechtman, A.~Beck, and Y.~C. Eldar}, {\em {GESPAR: Efficient phase
  retrieval of sparse signals}}, IEEE Transactions on Signal Processing, 62
  (2014), pp.~928--938, \url{https://doi.org/10.1109/TSP.2013.2297687},
  \url{https://arxiv.org/abs/1301.1018}.

\bibitem{Shechtman2015}
{\sc Y.~Shechtman, Y.~C. Eldar, O.~Cohen, H.~N. Chapman, J.~Miao, and
  M.~Segev}, {\em {Phase Retrieval with Application to Optical Imaging: A
  contemporary overview}}, IEEE Signal Processing Magazine, 32 (2015),
  pp.~87--109, \url{https://doi.org/10.1109/MSP.2014.2352673},
  \url{https://arxiv.org/abs/1402.7350}.

\bibitem{Soulez2017}
{\sc F.~Soulez, {\'{E}}.~Thi{\'{e}}baut, A.~Schutz, A.~Ferrari, F.~Courbin, and
  M.~Unser}, {\em {Proximity Operators for Phase Retrieval}},  (2017),
  pp.~1--8, \url{https://doi.org/10.1364/AO.55.007412},
  \url{http://arxiv.org/abs/1710.10046 http://dx.doi.org/10.1364/AO.55.007412},
  \url{https://arxiv.org/abs/1710.10046}.

\bibitem{Waldspurger2013}
{\sc I.~Waldspurger, A.~D'Aspremont, and S.~Mallat}, {\em {Phase recovery,
  MaxCut and complex semidefinite programming}}, Mathematical Programming, 149
  (2013), pp.~47--81, \url{https://doi.org/10.1007/s10107-013-0738-9},
  \url{https://arxiv.org/abs/1206.0102}.

\end{thebibliography}
\end{document}